\documentclass[a4paper,fleqn,leqno,11pt]{amsart}
\usepackage{amsfonts,latexsym,amsmath,amsthm,amscd,eucal,amssymb, dsfont} 
\usepackage[utf8]{inputenc}
\usepackage{color}
\usepackage[normalem]{ulem}
\usepackage{tikz}
  \usetikzlibrary{calc,backgrounds,positioning,shadows,fadings}
\usepackage{graphicx}%, mathscinet}

\def\ps@myheadings{\let\@mkboth\@gobbletwo
  \def\@oddhead{{\slshape\rightmark}\hfil{\footnotesize\thepage}}%{\hfill\hbox{}\rightmark}
  \def\@oddfoot{}
  \def\@evenhead{{\footnotesize\thepage}\hfil\slshape\leftmark}%{\leftmark\hbox{}\hfill}
  \def\@evenfoot{}
  \def\sectionmark##1{}\def\subsectionmark##1{}
}

\oddsidemargin=\evensidemargin
\def\runninghead#1#2{\pagestyle{myheadings}
\markboth{{\protect\footnotesize\it{\quad #1}}}
{{\protect\footnotesize\it{#2\quad}}}}
\newtheorem{theorem}{Theorem}[section]
\newtheorem{lemma}[theorem]{Lemma}
\newtheorem{proposition}[theorem]{Proposition}
\newtheorem{corollary}[theorem]{Corollary}

\newtheorem{definition}[theorem]{Definition}

\newtheorem{remark}[theorem]{Remark}
\numberwithin{equation}{section}

\def\cB{\mathcal B}

\def\cE{\mathcal E}
\def\cF{\mathcal F}

\def\cH{\mathcal H}

\def\cS{\mathcal S}

\def\bD{\mathbb D}
\def\bE{\mathbb E}

\def\bP{\mathbb P}

\def\bR{\mathbb R}

\newcommand{\slantfrac}[2]
{\mbox{
${}^#1{\mskip -5mu/\mskip -3mu}_{#2}$
}}
\def\uno{\mathds{1}}

\begin{document}

%
%%%%%%%%%%%%%%%%%%%%%%%%%%%%%%% GEOMETRIA
%
\pagestyle{plain}
\thispagestyle{empty}
\setlength{\textheight}{7.7truein}    %FOR 2ND PAGE ONWARDS
\runninghead{S.\ Bonaccorsi and L.\ Tubaro and M.\ Zanella}{Surface measures and i.b.p.\ formula on levels sets induced by the Brownian motion in $\bR^n$}

%
%%%%%%%%%%%%%%%%%%%%%%%%%%%%%%% INTESTAZIONE
%
\hfill{December 22, 2018}
\par\vskip 2\baselineskip
\par\rm
%
%%%%%%%%%%%%%%%%%%%%%%%%%%%%%%% TITOLO
%
\noindent{\LARGE \bf Surface measures and integration by parts \\ formula  on levels sets induced by functionals \\ of the Brownian motion in $\bR^n$}
\par\vskip 2\baselineskip
\par
%
%%%%%%%%%%%%%%%%%%%%%%%%%%%%%%% AUTORI
%
\noindent{\large \sc Stefano Bonaccorsi$^{a}$, Luciano Tubaro$^{a}$ and Margherita Zanella$^{b}$
}
\par
\vskip 1\baselineskip
\noindent{$^{a}$Dipartimento di Matematica, Universit\`a di Trento \\ via Sommarive 14, 38123 POVO TN, Italy \\ E-mail: {\tt stefano.bonaccorsi@unitn.it}, {\tt luciano.tubaro@unitn.it}}
\par\vskip1\baselineskip
\noindent{$^{b}$Dipartimento di Economia e Finanza, Universit\`a LUISS Guido Carli \\ viale Romania 32, Roma, Italy \\ E-mail: \tt mzanella@luiss.it}

\par\vskip1\baselineskip

\par
\vskip 1\baselineskip

%
%%%%%%%%%%%%%%%%%%%%%%%%%%%%%%% DATA
%
%\par\hfill{\sl \today}
%\par
%\vskip 1\baselineskip
%
%%%%%%%%%%%%%%%%%%%%%%%%%%%%%%% FINE INTESTAZIONE

\begin{center}
\begin{minipage}{.8\textwidth}
\section*{Abstract}
{On the infinite dimensional space $E$ of continuous paths from $[0,1]$ to $\bR^n$, $n \ge 3$, endowed with the Wiener measure $\mu$,
we construct a surface measure defined on level sets of the $L^2$-norm of  $n$-dimensional processes that are solutions to a general class of stochastic differential equations, and provide an integration by parts formula involving this surface measure.
%\textcolor{blue}{
%The construction of a surface measure in infinite dimensional spaces is an important topic in literature.} \textcolor{red}{(Io non saprei fare una breve panoramica della letteratura esistente e dei diversi approcci, volete scrivere qualcosa tu e Luciano?)}
We follow the approach to surface measures in Gaussian spaces proposed via techniques of Malliavin calculus in \cite{Airault1988}.}
\end{minipage}
\end{center}

\par
\vskip 1\baselineskip

\section{Introduction}

Let $E = C([0,1]; \bR^n)$ denote the Banach space of continuous functions from $[0,1]$ to $\bR^n$, endowed with the sup-norm $\displaystyle \|f\|_\infty = \sup_{[0,1]} |f(x)|$. We denote by $\cE = \cB(E)$ the $\sigma$-field of Borel measurable subsets of $E$.
Also, we introduce the Hilbert space $H = L^2(0,1;\bR^n)$ of square integrable measurable functions.
\\
Let us fix the notation we shall use in the sequel. The norm in $\bR^n$ is denoted by $|x|$ and the scalar product as $\langle x,x \rangle_{\bR^n}$.
The (equivalent) $L^1$-norm in $\bR^n$ is $|x|_1 = \sum |x_i|$.
 In the infinite dimensional spaces $E$ and $H$ we denote the norm respectively by $\|x\|_H$, $\|x\|_E$. Finally, the scalar product in $H$ is $\langle x,x \rangle_H$. By $E^*$ we we denote the dual of $E$.

It is known that given a probability space $(\Omega, \cF, \bP)$, a process $B = \{B(t),\, t \in [0,1]\}$  is a standard $n$-dimensional Brownian motion if it a centered Gaussian process with covariance function $\bE[\langle B(t), B(s)\rangle] = (s \wedge t)I$, $I$ being the identity matrix in $\bR^n$.
This process induces a Gaussian measure $\mu$ on the space of trajectories $(E, \cE)$. This measure is known as the {\em Wiener measure};  the process
$B(t)(x) = x(t)$ on the probability space $(E, \cE, \mu)$ will be denoted the standard $n$-dimensional Brownian motion.

On the space $E$ we introduce the Malliavin derivative $D$ with domain $\bD^{1,p}$ (that is the closure in $L^p(E, \mu)$ of the class of smooth random variables) see for instance \cite{Nualart2006, Bogachev1998, DaPrato2014a, DaPrato2014}. 
In Section \ref{s.malliavin} below we explain its construction in more details.
The adjoint operator of the Malliavin derivative operator $D$ having domain $\bD^{1,p}$ is the divergence operator, denoted as usual by $\delta$, having domain   $D_q(\delta)$, where $q = p'$ is the adjoint exponent of $p$.
$\delta$ coincides with the Skorohod integral with respect to the Brownian motion $B$. 

In addition to the Sobolev spaces $\bD^{1,p}$, we shall consider the spaces $UC_b(E)$ of uniformly continuous and bounded functions\footnote{%
we use indifferently the terms function, functional or random variable to denote a measurable mapping $F: E \to \bR$} % 
from $E$ to $\bR$ and $UC_b^1(E)$, of uniformly continuous and bounded functions which are Fr\'echet differentiable, with an uniformly continuous and bounded derivative.
\\
In the sequel, we simplify the notation to $UC_b$, $UC_b^1$, since no confusion may arise.

\vskip 1\baselineskip

Let $u \in L^p(E,\mu;H)$ be a stochastic process, indexed by $t \in [0,1]$, taking values in $\bR^n$. The simplest example of such processes is, obviously,
the Brownian motion $B$:
\begin{align*}
B(t)(x) = x(t), \qquad x \in E,\ t \in [0,1].
\end{align*}
In this paper we introduce the functional $g:L^p(E,\mu;H) \rightarrow \mathbb{R}$ which associates to any such process the random variable
\begin{align}\label{eq:gintro}
g(u)(x) = \frac12 \|u(x)\|^2_H = \frac12 \, \int_0^1 |u(x)(t)|^2 \, {\rm d}t.
\end{align}
In case $u = B$, we shall simply write $g(x) = g(B)(x) = \frac 12\|x\|^2_H$.
\newline 
The aim of this paper is to construct the surface measure induced by $\mu$ on the level sets $\{ g = r \}$ {and provide an integration by parts formula involving this surface measure}. 
We shall mention here that, since the domain $\{g < r\}$ is a convex open set in $E$, our construction is related to that of the recent paper \cite{Addona2018}.
In particular, the integration by parts formula that we obtain in Proposition \ref{pr:ibpf} is related to formula (1) in \cite{Addona2018}. 
Notice however that our construction is quite different. For instance, they choose the measure $\sigma$ on the level sets of $g$ by appealing to the construction of \cite{Feyel1992} to fix a reference surface measure to use in the integration by parts formula.
On the other hand,  we construct the measure $\sigma$ by following the approach initiated by Airault and Malliavin \cite{Airault1988}.
\\
Let $X \in L^1(E,\mu)$ be a random variable (more stringent assumptions on $X$ will be necessary, compare Section \ref{sez3}). Then we define the function
\begin{align}\label{e:i2}
F_X(r) = \int_{\{g(u) < r\}} X(x) \, \mu({\rm d}x), \qquad r \in \bR; %, \qquad X \in L^1(E,\mu);
\end{align}
if $F_X$ is differentiable at $r$, its derivative $F'_X(r)$ is candidate to be a surface integral
\begin{align}\label{e:i3}
F'_X(r) = \int_{\{g(u) = r\}} X(x) \, \sigma_r({\rm d}x),
\end{align}
provided that there exists a measure $\sigma_r$, independent of $X$, such that \eqref{e:i3} holds.
Obviously, one further needs to prove that $\sigma_r$ is concentrated on $\{ g = r\}$. % and to clarify the dependence on $g$.
This approach was followed, among others, by \cite{Bogachev1998, DaPrato2014, Bogachev2016, Bonaccorsi2018, DaPrato2018}.
The main result in this paper is given in the the following theorem, whose proof is given in Section \ref{sez2}. %As we shall see, existence of such objects is the key point in the proof of our main theorem \ref{}).

\begin{theorem}\label{t1}
Let %$(\Omega, \mathcal{F}, \mathbb{P})$ be a probability space and 
$B$ the standard $n$-dimensional Brownian motion defined on %it
the Wiener probability space $(E, \cE, \mu)$.
Assume that the dimension $n$ %of the reference Brownian motion $B$ 
satisfies
\begin{equation}
n \ge 3.
\end{equation}
Let $g$ be the random variable defined above
\begin{align*}
g(x) = g(B)(x) = \frac12 \|x\|^2_H, \qquad x \in E,
\end{align*} %in \eqref{e:i1} 
and consider the function $F_X$ defined in \eqref{e:i2}.
Then, for any $r > 0$ there exists a unique Borel measure $\sigma_r$ on $E$ such that \eqref{e:i3} holds for any $X \in UC_b \cup \bD^{1,p}$ and  the support of $\sigma_r$ is concentrated on $\{g = r\}$.
\newline
{Moreover, for fixed $r > 0$, for any $X \in \bD^{1,p}$ and $h \in H$, the following integration by parts formula holds
\begin{equation*}
\int_{\{g < r\}}  \langle DX,h \rangle_H  \, \mu({\rm d}x) =
- \int_{\{g=r\}} X \langle D g, h \rangle_H \, \sigma_r({\rm d}x) + \int_{\{g < r\}}  X W(h)  \, \mu({\rm d}x),
\end{equation*}
with $W(h)$ the Gaussian random variable defined in \eqref{e:wie1}. 
}
\end{theorem}

It is necessary to emphasize that the main effort in the proof is required by proving 
the following proposition, which states that the random variable $g$ satisfies, in a suitable sense, the local Malliavin condition, see \cite{Bonaccorsi2018}.
Such condition was introduced by Nualart \cite[Definition 2.1.2]{Nualart2006} (in a slightly different formulation) in a related context, i.e., the analysis of the density for the law of a random variable. \textcolor{black}{In our construction the law of the random variable $g$ plays a crucial role since it provides an explicit characterization of the surface measure $\sigma_r$ (see \eqref{charact}).}

\begin{proposition}\label{hp:1}[Malliavin condition on $g$]
%Assume that there 
There exists a process $u \in L^p(E,\mu:H)$ %{\color{red}$u: \Omega \to H$} 
and a real valued random variable $\gamma \in \bD^{1,p}$, for any $p > 1$, %{\color{red}$\gamma: \Omega \to \bR$} 
such that the following identity holds
\begin{equation}
\label{e:i4}
\langle D g, u \rangle_H = \gamma
\end{equation}
and $\frac{u}{\gamma}$ belongs to $D_q(\delta)$ for every
\begin{equation}
\label{e:boundq}
1 < q < \tfrac{n}2.
\end{equation}
\end{proposition}

The proof of this Proposition is given in Section \ref{sez3}. The proof of the main theorem, which is given in Section \ref{sez2},
follows quite naturally by the ideas provided in \cite{Bonaccorsi2018, DaPrato2018}.

{\color{black}Our construction, in particular, leads to the following identity concerning the surface measure $\sigma_r$:
\begin{equation}
\label{charact}
\bE[X \mid g=r] \, f_1(r) = \int_{\{g=r\}} X(\xi) \, \sigma_r({\rm d}\xi),
\end{equation}
where $f_1$ is the probability density function of the random variable $g = g(B)$. 
Corollary \ref{c2.3} below assures that $f_1$ is a bounded and continuous function and the identity above holds for every $r > 0$.}

In the last part of the paper we extend previous results to the analysis of the random variables $g(u)$, where we assume that $u$ is  the
solution of a stochastic differential equation of gradient form
\begin{equation}
\label{e:sde-intro}
u(t) = - \int_0^t \nabla V(u(s)) \, {\rm d}s + B(t).
\end{equation}
{This is the first step in considering processes whose image law is non Gaussian.}
In particular, in Section \ref{sez5} we prove the following result.

\begin{theorem}
Let $u$ be the solution of equation \eqref{e:sde-intro}, where $V \in C^3_b(\bR^n;\bR^n)$, $n \ge 3$, and $g(u)$ be the random variable defined in \eqref{eq:gintro}.
\\
Then $g(u)$ defined on $(E, \cE, \mu)$, has a continuous and bounded density $\varphi_1$ with respect to the Lebesgue measure on $\bR_+$.
Moreover, there exists a surface measure $\theta_r$  concentrated on $\{ g(u) = r\}$ 
that is the restriction of $\mu$ to the level set $\{ g(u) = r\}$ .
\end{theorem}

We notice that the probability density function $\varphi_1(r)$ of the random variable $g(u)$ with respect to the Lebesgue measure can be computed in terms of $f_1$ as follows:
\begin{align*}
\varphi_1(r) = \bE[ \rho_1(B)^{-1} \mid g(B) = r] f_1(r),
\end{align*}
where $\rho_1(B)^{-1}$ is a bounded function which is defined in terms of the coefficient $V$ in \eqref{e:sde-intro}.
\\
Moreover, it follows that $\varphi_1(r) = \theta_r(\{g(u)=r\})$.
The proof is based on a Girsanov transformation of the reference Gaussian measure and it exploits the results obtained in the case $u=B$.

\section{An introduction to Malliavin calculus}
\label{s.malliavin}

{
In literature different ways of introducing the Malliavin derivative are present. We work here in the general framework given in \cite{Nualart2006}. This approach requires to fix a probability space $(\Omega, \mathcal{F}, \mathbb{P})$ and an isonormal Gaussian process which provides the Gaussian framework.
Here, as reference probability space, we consider the Wiener space $(E, \mathcal{E}, \mu)$. The isonormal Gaussian process is given by the family of Wiener integrals.
}

We denote, as before, by $B(t)(x) = x(t)$, $ t \in [0,1]$, $x \in E$, the standard $n$-dimensional Brownian motion on the probability space $(E, \cE, \mu)$.
Given this process, we may introduce the Wiener integral
\begin{align}\label{e:wie1}
W(h) = \int_0^1 \langle h(s), {\rm d}B(s) \rangle_{\bR^n}, \qquad h \in H = L^2(0,1;\bR^n).
\end{align}
For any $h \in H$, $W(h)$ is a centered Gaussian random variable with variance $\|h\|^2_H$.
We shall denote by $\cH_1$ the following subspace of $L^2(E,\mu)$, called the {\em first Wiener chaos},
defined by
\begin{align*}
\cH_1 = \{F \in L^2(E,\mu)\,:\, \exists\, h \in H,\ F = W(h)\}.
\end{align*}
{
The map $W$ defines a linear isometry between $H$ and $\mathcal{H}_1$. In particular, we have $\{B(t),\, t\in \left[0,1\right]\} = \{W(\pmb{1}_{\left[0,t\right]}),\, t\in \left[0,1\right]\}$.
}
\begin{remark}
\em In the sequel, we shall use  the probabilistic notation of expectation for the integral over $E$
\begin{align*}
\bE[F] = \int_E F(x) \, \mu({\rm d}x),
\end{align*}
for a measurable function (random variable) $F: (E,\cE) \to \bR$. In particular,
\begin{align*}
\bE[W(h)] = \int_E W(h)(x) \, \mu({\rm d}x) = 0, \qquad \bE[|W(h)|^2] = \int_E |W(h)(x)|^2 \, \mu({\rm d}x) = \|h\|^2_H.
\end{align*}
\end{remark}

Starting from the space $\cH_1$ we construct the class of smooth random variables
\begin{align*}
\cS = \{F : (E, \cE) \to \bR\,:\, \exists\, f\in C^\infty_P(\bR^d),\ h_1, \dots, h_d \in H,\ F = f(W(h_1), \dots, W(h_d)\},
\end{align*}
where $C^\infty_P(\bR^d)$ is the space of smooth functions on $\bR^d$ with polynomial growth at infinity.
\\
We see that $\cS \subset L^p(E,\mu)$ for any $p \ge 1$. On the class $\cS$ of smooth random variables we consider a functional (actually, a family of functionals indexed by the order of integration $p$)
\begin{align*}
D: \cS \subset L^p(E,\mu) \to L^p(E,\mu;H)
\end{align*}
by setting
\begin{align*}
D F = \sum_{k=1}^d \frac{\partial f}{\partial x_k}(W(h_1), \dots, W(h_d)) \, h_k.
\end{align*}

\begin{lemma}
Let $F \in \cS$ and $h \in H$. Then it holds
\begin{align*}
\bE[\langle DF, h \rangle_H] = \bE[F W(h)].
\end{align*}
\end{lemma}

For the proof we refer to \cite{Naualrt2006}.
As a consequence, it is possible to prove that the operator $D$ is closable from $L^p(E,\mu)$ to $L^p(E,\mu;H)$.

\begin{definition}
We define the norm
\begin{align*}
\|F\|_{1,p}^p = \bE[|F|^p] + \bE[\|D F\|^p_H], \qquad F \in \cS.
\end{align*}
Then the {\em domain} of the Malliavin derivative $D$, denoted by $\bD^{1,p}$, is the closure of the class $\cS$ in $L^p(E,\mu)$ with respect to the norm $\|\cdot\|_{1,p}$. We shall denote again by $D$ this closure.
\end{definition}

\begin{center}
\begin{picture}(200,100)
\put(0, 90){$H$}
\put(10, 93){\vector(1,0){35}}
\put(20, 98){\small $W$}
\put(50, 90){$\cH_1$}
\put(70, 90){$\subset$}
\put(88, 90){$\cS$}
\put(95, 93){\color{black!50!white}\vector(1,-1){80}}
\put(175, 10){$L^p(E,\mu;H)$}
\put(88, 70){\rotatebox[origin=c]{-90}{$\subset$}}
\put(88, 50){$\overline{\cS}^{1,p}$}
\put(88, 30){\rotatebox[origin=c]{-90}{$=$}}
\put(88, 10){$\bD^{1,p}$}
\put(30, 10){$L^{p}(E,\mu) \supset$}
\put(118, 58){\color{black!50!white}\rotatebox[origin=c]{-45}{\small $D W(h) = h$}}
\put(110, 13){\vector(1,0){65}}
\put(130, 18){\small $D$}
\end{picture}
\end{center}

\bigskip

{\color{black}

Let us now introduce the divergence operator.
\newline
Fix $1<q<\infty$ with $\frac 1p+\frac 1q=1$. By $D_q(\delta)$ we denote the domain of the diverge operator $\delta$. It consists of all $v \in L^q(E, \mu,;H)$ for which there exits a $G_v \in L^q(E, \mu)$ such that 
\begin{equation*}
\mathbb{E} \left[ \langle v, DF\rangle_H \right]= \mathbb{E} \left[G_vF \right], \qquad F \in \mathbb{D}^{1,p}.
\end{equation*}
The function $G_v$, if it exists, is uniquely determined. We set
\begin{equation*}
\delta(v):=G_v, \qquad v \in D_q(\delta).
\end{equation*}
The divergence operator is easily seen to be closed and densely defined. 

It is known that in the case $H=L^2(0,1;\mathbb{R}^n)$, for a given $F \in \mathbb{D}^{1,p}$, its Malliavin derivative $DF \in L^p(E, \mu;H)$ can be interpreted as a stochastic process indexed by $t \in \left[0,1\right]$. In this case we can interpret the divergence operator as a stochastic integral, the Skorohod integral and the following notation becomes significant: 
\begin{equation*}
\delta(u)=\int_0^1 \langle u(s), \delta B(s)\rangle_{\mathbb{R}^n}.
\end{equation*}
If the process $u$ is adapted and It\^o integrable then the Skorohod integral coincides with the It\^o integral. In the special case $u = h \in H$, we have $\delta(u)=W(u)$, with $W(u)$ given by \eqref{e:wie1}.

We conclude this Section with the following proposition, that we will exploit to prove our main result Theorem \ref{t1}. It is an extension of \cite[Proposition 1.3.3]{Nualart2006} and a proof can be found in \cite[Proposition 6.9]{Maas2010}.

\begin{proposition}\label{N}
Let $1 < r,q < \infty$ be such that $\frac1q = \frac1p + \frac1r$. Assume that $F \in \bD^{1,p}$ and $u \in D_{r}(\delta)$.
Then $Fu \in D_{q}(\delta)$ and
\begin{align*}
\delta(Fu) = F \delta(u) - \langle u, DF \rangle_H.
\end{align*}
\end{proposition}
}

\bigskip

\section{Verification of the Malliavin condition}\label{sez3}

In this section we prove that the random variable
\begin{equation}
\label{g}
g(x) = \frac12 \|x\|^2_H = \frac12 \int_0^1 |x(t) |^2\, {\rm d}t, \qquad x \in E,
\end{equation}
satisfies the Malliavin condition stated in Proposition \ref{hp:1}.
From a probabilistic point of view, this random variable is strictly related to the Bessel process of order $\nu = \frac{n}{2}-1$: $X(t) = |B(t)|^2$. Some results on $g$ are given, for instance, in \cite{Borodin2002}.

In the following proposition we state the construction that we aim to prove in order to achieve the verification of the Malliavin condition.

\begin{proposition}
\label{hpg}
Assume that $B$ is the $n$-dimensional standard Brownian motion with $n \ge 3$.
The random variable $g$, defined in \eqref{g}, satisfies the Malliavin condition in Proposition \ref{hp:1} with
\begin{equation}\label{e:prop32}
u(s)=\frac{D_sg}{|D_sg|}, \qquad \gamma= \int_0^1 |D_sg | \,{\rm d}s.
\end{equation}
\end{proposition}

\smallskip

\begin{remark}
\em As we shall see in the proof of Lemma \ref{lem1.5}, the condition on $n$ stated in Theorem \ref{t1} comes from the estimate of $\bE[\gamma^{-1/p}]$. 
The condition $n \ge 3$ is therefore {\em sufficient} for our construction to hold. It remains open the problem of whether it is also {\em necessary}. 
Notice that the existence of the probability density function $f_1$ for the random variable $g$ holds (via different techniques) for every $n \ge 1$.
\end{remark}
The proof of Proposition \ref{hpg} is based on a chain of four lemmas that we will prove in the next Subsection.

\subsection{Preliminary lemmas}

\begin{lemma}
\label{derg}
Let $\{B(t),\, t \in \left[0,1\right]\}$ be the standard $n$-dimensional Brownian motion. The function $g=g(B)$ as defined in \eqref{g} belongs to the space $\mathbb{D}^{1,p}$ for all $p \ge 1$ and its derivative (in the direction $i$) is given by
\begin{equation}
D_s^ig=\int_s^1B^i(t)\, {\rm d}t.
\end{equation}
\end{lemma}

\begin{proof}
For every $ s \in \left[0,1\right]$ we can compute the Malliavin derivative, in the direction $e_i$, of the function $g$ as follows
\begin{equation}\label{e.dg}
D_s^ig:=\langle D_sg , e_i\rangle_{\mathbb{R}^n}= \int_0^1 B^i(t)D_s(B_s^i(t))\, {\rm d}t=\int_s^1B^i(t)\, {\rm d}t.
\end{equation}
\end{proof}

\begin{lemma}
\label{Lem1}
Let $\{B(t),\, t \in \left[0,1\right]\}$ be the standard $n$-dimensional Brownian motion and let $\gamma$ be defined as in Proposition \ref{hpg}.
Then for any $0 < \eta < 1$ there exists a constant $c$ such that 
\begin{equation}
\mathbb{P} \left( %\int_0^1 \left(\sum_{i=1}^n \left| \int_s^1B^i(t)\, {\rm d}t\right|^2\right)^{p}\, {\rm d}s 
\gamma < \eta \right) \le c \, \eta^{n}.
\end{equation}
\end{lemma}

\begin{proof}
We compute, for $\eta \ll 1$,
\begin{align*}
\bP(\gamma < \eta) = 
%\bP\left( \int_0^1 \left[ \sum_{i=1}^n \left( \int_0^1 \uno_{(0,t)}(s) B(t) \, {\rm d}t \right)^2 \right]^{\frac12} \, {\rm d}s < \eta \right)
\bP\left( \int_0^1 |D_sg | \, {\rm d}s < \eta \right).
\end{align*}
Notice that every two norms in $\bR^n$ are equivalent, so there exists $c$ such that
\begin{align*}
\bP(\gamma < \eta) = 
%\bP\left( \int_0^1 \left[ \sum_{i=1}^n \left( \int_0^1 \uno_{(0,t)}(s) B(t) \, {\rm d}t \right)^2 \right]^{\frac12} \, {\rm d}s < \eta \right)
\bP\left( \int_0^1 |D_sg |_1 \, {\rm d}s < c \, \eta \right). %\qquad \text{where $\displaystyle |y|_1 = \sum_{i=1}^n |y_i|$.}
\end{align*}
Further, by H\"older's inequality,
\begin{align*}
\bP(\gamma < \eta) \le 
\bP\left( \left| \int_0^1 D_sg \, {\rm d}s \right|_1 < c \, \eta \right)
\end{align*}
and, using again the equivalence of norms, we find a different $c$ such that
\begin{align*}
\bP(\gamma < \eta) \le 
\bP\left( \left| \int_0^1 D_sg \, {\rm d}s \right| < c \, \eta \right)
\end{align*}
Notice that
\begin{align*}
\int_0^1 D^i_sg \, {\rm d}s =& \int_0^1 \int_0^1 \uno_{(0,t)}(s) B^i(t) \, {\rm d}t \, {\rm d}s
= \int_0^1 \int_0^1 \uno_{(0,t)}(s) \, {\rm d}s B^i(t) \, {\rm d}t
\\
=& \int_0^1 t B^i(t) \, {\rm d}t = \frac12 \int_0^1 (1 - t^2) \, {\rm d}B^i(t) =: Z_i
\end{align*}
is a family of independent, identically distributed Gaussian random variables with zero mean and variance $\sigma^2 = \int_0^1 \left(\frac{(1-t^2)}{2}\right)^2 \, {\rm d}t$.
Therefore, for some constant $c$ depending on $\sigma$,
\begin{align*}
\bP(\gamma < \eta) \le 
\bP\left( \sum_{i=1}^n \left(\tfrac1\sigma Z_i \right)^2 < c \eta^{2} \right).
\end{align*}
Since $Z_i$ are independent random variables, the sum in the right-hand side is a $\chi^2$-distribution with $n$ degrees of freedom, hence
the probability above is estimated by
\begin{align*}
\bP(\gamma < \eta) \le c \left(\eta^2\right)^{\frac n2} = c \eta^n.
\end{align*}
\end{proof}

\begin{lemma}
\label{lem1.5}
Let $\{B(t),\, t \in \left[0,1\right]\}$ be the standard $n$-dimensional Brownian motion and let $\gamma$ be defined as in Proposition \ref{hpg}.
Then 
$\mathbb{E}\left|\slantfrac {1}{\gamma}\right|^p< \infty$ for every $1 < p < n$.
\end{lemma}

\begin{proof}
Fixed $p>1$, we need to show that $\mathbb{E} \left|\slantfrac{1}{\gamma}\right|^p < \infty$.
We can write 
\begin{align*}
\mathbb{E} \left|\slantfrac{1}{\gamma}\right|^p 
&= \int_0^{\infty}\mathbb{P} \left(\frac{1}{\gamma^p}>\theta\right) \, {\rm d}\theta
=\int_0^{\infty}\mathbb{P} \left(\gamma< \frac{1}{\theta^{\frac 1p}}\right) \, {\rm d}\theta
=p\int_0^{\infty}\mathbb{P} \left(\gamma< \frac{1}{\tau}\right) \tau^{p-1} \, {\rm d}\tau
\\
&=p\int_0^1\mathbb{P} \left(\gamma< \frac{1}{\tau}\right) \tau^{p-1} \, {\rm d}\tau
+
p\int_1^{\infty}\mathbb{P} \left(\gamma< \frac{1}{\tau}\right) \tau^{p-1} \, {\rm d}\tau.
\end{align*}
The first integral in the last line of the above expression is finite since $\mathbb{P}$ is a probability measure and thus $\mathbb{P} \left(\gamma< \frac{1}{\tau}\right)\le 1$. Thus it is sufficient to show the convergence of the second integral.

%Since
%\begin{equation*}
%\gamma=\int_0^1 \left(\sum_{i=1}^n \left| \int_s^1B^i(t)\, {\rm d}t\right|^2\right)^{\frac 12}\, {\rm d}s,
%\end{equation*}
Thanks to Lemma \ref{Lem1} we immediately get
\begin{align*}
\int_1^{\infty}\mathbb{P} \left(\gamma< \frac{1}{\tau}\right) \tau^{p-1} \, {\rm d}\tau
&\le C \int_1^{\infty}\tau^{- n+p-1}\,{\rm d}\tau,
\end{align*}
which is finite provided $n>p$.
\end{proof}

\begin{lemma}
\label{lem1.6}
Let $\{B(t),\, t \in \left[0,1\right]\}$ be the standard $n$-dimensional Brownian motion and let $\gamma$ be defined as in Proposition \ref{hpg}.
Then $\gamma \in \mathbb{D}^{1,p}$ for every $p > 1$ and its Malliavin derivative (in the direction $i$) is
\begin{equation}
\label{D0}
D_{\theta}^i\gamma=\int_0^1 \frac{1}{|D_sg |}\int_s^1B^i(r)\, {\rm d}r\int_s^1\uno_{(0,t)}(\theta)\, {\rm d}t\, {\rm d}s.
\end{equation}
Moreover $\slantfrac {1}{\gamma} \in \mathbb{D}^{1,p}$ for any $1 < p < \frac{n}2$ and its Malliavin derivative is given by
\begin{equation}
\label{D1}
D_{\theta}^i\left(\slantfrac{1}{\gamma}\right)=- \frac{D^i_{\theta} \gamma}{\gamma^2}.
\end{equation}
\end{lemma}

\begin{proof}
By the chain rule and Lemma \ref{derg} we have that $\gamma \in \mathbb{D}^{1,p}$ for every $p>1$ and 
\begin{align*}
D_{\theta}^i \gamma
&=D_{\theta}^i\left(\int_0^1 \left(\sum_{i=1}^n \left| \int_s^1B^i(t)\, {\rm d}t\right|^2\right)^{\frac 12}\, {\rm d}s\right)
= \int_0^1\frac{1}{ |D_sg |}\int_s^1B^i(r)\,{\rm d}r \int_s^1\uno_{(0,t)}(\theta)\, {\rm d}t \,{\rm d}s.
\end{align*}
Therefore,
applying  H\"older's inequality, from \eqref{D0} we get 
\begin{align*}
|D_{\theta}\gamma|^2
&=  \sum_{i=1}^n \left|  \int_0^1 \frac{1}{ |D_sg |}\int_s^1B^i(r)\, {\rm d}r\int_s^1\uno_{(0,t)}(\theta)\, {\rm d}t\, {\rm d}s \right|^2
\\
&\le \sum_{i=1}^n  \int_0^1 \frac{1}{ | D_sg |^2} \left| \int_s^1B^i(r)\, {\rm d}r\int_s^1\uno_{(0,t)}(\theta)\, {\rm d}t \right|^2 \, {\rm d}s, 
\end{align*}
{taking the sum inside the integral we can simplify and we get}
\begin{align}\label{e:Dgamma}
|D_{\theta}\gamma |^2 &\le 1.
\end{align}

Equality \eqref{D1} is straightforward to prove. 
It remains to prove that $\frac{1}{\gamma} \in \mathbb{D}^{1,p}$. 
In view of Lemma \ref{lem1.5} it is sufficient to show that
\begin{equation}
\label{D2}
\mathbb{E}\left\Vert D \left(\slantfrac{1}{\gamma}\right)\right\Vert_{H}^p< \infty.
\end{equation} 
By means of estimate \eqref{e:Dgamma} we get
\begin{align*}
\mathbb{E}\left\Vert D_{\cdot}\left(\frac{1}{\gamma}\right)\right\Vert_{H}^p
&=\mathbb{E}\left\Vert\left(\frac{D_{\cdot}\gamma}{\gamma^2}\right)\right\Vert_{H}^p
=\mathbb{E}\left[ \frac{\|D_{\cdot}\gamma\|_{H}}{|\gamma|^2}\right]^p
\le \mathbb{E} \left|\slantfrac{1}{\gamma}\right|^{2 p}.
\end{align*}
By Lemma \ref{lem1.5}, the term $\mathbb{E} \left|\slantfrac{1}{\gamma}\right|^{2 p}$ is finite provided that $n \ge 3 p$
and this concludes the proof.
\end{proof}

We see that the condition on the dimension $n$ appears in the statement of previous proposition, since it is necessary that $\frac{n}{2} > 1$, i.e., $n \ge 3$ in order to have a non degenerate interval for $p$.

We are now ready to prove Proposition \ref{hpg}.

\subsection{Proof of Proposition \ref{hpg}}

Thanks to Lemma \ref{derg} we know that $g \in \mathbb{D}^{1,p}$, for every $p > 1$. 
By definition, compare \eqref{e:prop32}, we verify that
\begin{equation}
\label{a}
\langle Dg, u \rangle_H = \int_0^1 \langle D_sg, u(s) \rangle_{\bR^n} \, {\rm d}s=\gamma.
\end{equation}
It remains to prove that 
\begin{align*}
\tfrac{u}{\gamma} \in D_{q}(\delta) \text{ for all $1 < q < \tfrac{n}{2}$.}
\end{align*}
In order to factor out a scalar random variable from a Skorohod integral we can appeal to Proposition \ref{N}. 
From Lemma \ref{lem1.6} we know that $\frac{1}{\gamma} \in \bD^{1,p}$ for every $1 < p < \frac{n}{2}$; 
we claim that $u \in D_{r}(\delta)$ for every $r > 1$. Therefore, by previous proposition, we have
$\frac{u}{\gamma} \in D_{q}(\delta)$ for $\frac1q = \frac1p + \frac1r$ and, since $r$ is arbitrary and $p < \frac{n}{2}$, we obtain that
\begin{align*}
\tfrac{u}{\gamma} \in D_{q}(\delta) \text{ for all $q < \tfrac{n}{2}$.}
\end{align*}

It remains to verify the claim. Since $u$ is a process adapted to the future and $|u|\le 1$, $u$ is backward It\^o integrable. 
Then, we can mimic the construction given in \cite[Section 1.3.3]{Nualart2006} for %$L^2_a$ of 
square integrable, adapted processes to be Skorohod integrable, and prove that $u$ belongs to $D_r(\delta)$
for any $r$.

\section{Proof of Theorem \ref{t1}: existence of the surface measure for sets defined by the Brownian motion}
\label{sez2}
\textcolor{black}
{
In this section, by mimicking the construction provided in \cite{Bonaccorsi2018} we construct the surface measure induced by $\mu$ on the level sets $\{g = r\}$.
}
Recall that $g$ is the random variable defined by
\begin{align*}
g(x) = \frac12 \|x\|^2_H, \qquad x\in E;
\end{align*}
by construction, $g \ge 0$. 
\textcolor{black}{Moreover, thanks to Proposition \ref{hpg}, $g$ satisfies the Malliavin condition of Proposition \ref{hp:1}.}
\newline
As stated in the Introduction, we study the family of functions $F_X(r)$ indexed by (suitably regular) random variables $X$, where
\begin{align*}
F_X(r) = \bE[\uno_{\{g \le r\}} X], \qquad r \ge 0.
\end{align*}
In this section, we first assume that  $X$ is a random variable in $\bD^{1,p}$, for some $p > \frac{n}{n-2}$ \textcolor{black}{(as it will be clear later on, this condition stems from the requirement $n \ge 3$ in Proposition \ref{hpg})}.
Let 
\begin{align*}
\phi(u) = \int_{0}^u  \uno_{\{y \le r\}} \, {\rm d}y.
\end{align*}
$\phi$ is a Lipschitz continuous function, hence it is possible to compute the Malliavin derivative
\begin{align}\label{e:i5}
D \phi(g) = \phi'(g) D g;
\end{align}
scalar multiplying both sides of \eqref{e:i5} with $u$ and $X$ implies, after a little algebra
\begin{align*}
X \, \left\langle D \phi(g), \tfrac{u}{\gamma} \right\rangle_H = \phi'(g) \, X = X \, \uno_{\{g \le r\}};
\end{align*}
thus, the duality relationship between Malliavin derivative and Skorohod integral leads to
\begin{align}\label{e:2.3}
F_X(r) = \bE[X \uno_{\{g \le r\}}] = \bE \left[X \, \left\langle D \phi(g), \tfrac{u}{\gamma} \right\rangle_H \right] = \bE \left[\delta\left(X \, \tfrac{u}{\gamma}\right) \, \phi(g) \right].
\end{align}
\textcolor{black}{Notice that in order for the last term in \eqref{e:2.3} to be well defined, we need to have $X\frac{u}{\gamma} \in D_{\theta}(\delta)$ for some $\theta$. We postpone the verification of this fact to Subsection \ref{dif_sec} and we start by considering the special case $X \equiv 1$. This case allows us to study the probability density function of $g$.}

\subsection{Existence of the probability density function for $g$}

By taking $X \equiv 1$, the above reasoning leads to the existence of a density for the cumulative distribution function of the random variable $g$,
as already proved by Nualart \cite[Proposition 2.1.1]{Nualart2006}:
\begin{align*}
\bP(g \le r) = F_1(r) =& \bE \left[ \delta\left(\tfrac{u}{\gamma}\right) \, \int_{0}^r \uno_{\{y < g\}} \, {\rm d}y  \right]
\end{align*}
{and, by an application of Fubini's theorem, we get the following expression, which easily led to the existence of a density:}
\begin{align}\label{eq:numero}
F_1(r) =& \int_{0}^r  \bE \left[ \delta\left(\tfrac{u}{\gamma}\right)  \uno_{\{g > y\}}  \right] \, {\rm d}y.
\end{align}

\begin{proposition}
\label{p41}
Assume that there exist $u$ and $\gamma$ such that Proposition \ref{hp:1} holds.
Then the mapping $s \mapsto F_1(s)$ is continuous.
\end{proposition}

\begin{proof}
The integrand function $G: y \mapsto \bE \left[ \delta\left(\tfrac{u}{\gamma}\right)  \uno_{\{g > y\}}  \right]$, defined in \eqref{eq:numero}, is measurable and bounded, by assumption, hence the
statement is obvious.
\end{proof}

As a consequence of previous proposition, the mapping $G$ is also continuous, since
\begin{align*}
|G(y + \varepsilon) - G(y)| \le& \bE \left[ \left| \mathop\delta\left(\tfrac{u}{\gamma}\right)  \right| \uno_{\{y < g < y + \varepsilon\}} \right]
\\
\le& \bE \left[ \left| \delta\left(\tfrac{u}{\gamma}\right)  \right|^{q} \right]^{1/q} \left(F_1(y+\varepsilon) - F_1(y) \right)^{1/q'},
\end{align*}
where $q'$ is the conjugate exponent of $q$;
therefore, we can apply the integral mean value theorem to get the following.

\begin{proposition}\label{p2.1}
There exists the derivative $f_1(r) = F_1'(r)$ and it is equal to 
\begin{align*}
F_1'(r) = \bE \left[ \delta\left(\tfrac{u}{\gamma}\right)  \uno_{\{g > r\}}  \right]
\end{align*}
\end{proposition}

\begin{proof}
Since
\begin{align*}
\frac{1}{\varepsilon}\left( F_1(r + \varepsilon) - F_1(r) \right) = \frac{1}{\varepsilon} \int_r^{r + \varepsilon} G(y) \, {\rm d}y
\end{align*}
and the integrand function $G$ is continuous, the thesis follows by letting $\varepsilon \to 0$.
\end{proof}

\begin{corollary}\label{c2.3}
The random variable $g$ has a probability density function $f_1(r)$ that is continuous and bounded.
\end{corollary}

Actually, the existence of this density is already known in the literature, as well as the explicit form of this function, see \cite[Part II.4, formula (1.9.4), page 377]{Borodin2002}.

\subsection{Differentiability of $F_X$}
\label{dif_sec}
%\textcolor{red}{Ho riscritto questa parte come Lemma.}
\begin{lemma}
\label{l46}
Let $n \ge 3$. If $X \in \mathbb{D}^{1,p}$, for some $p>\frac{n}{n-2}$ and $\frac{u}{\gamma} \in D_q(\delta)$, with $q < \frac{n}{2}$ (cfr.\ Proposition \ref{hp:1}) then $X\frac{u}{\gamma} \in D_{\theta}(\delta)$ for some $\theta < \frac{np}{n+2p}$.
\end{lemma}
%We first prove that $X \frac{u}{\gamma}$ belongs to $D_{\theta}(\delta)$ , $p$ being the order of integrability of $X$ and $n$ the space dimension of the Brownian motion $B$. 
\begin{proof}
By definition, this requires to show that for any smooth random variable $Y$
\begin{align*}
\left|\bE\left\langle D Y, X \tfrac{u}{\gamma} \right\rangle_H \right| < c \, \bE[|Y|^{\theta'}]^{1/\theta'}, 
\end{align*}
with $\frac{1}{\theta}+ \frac{1}{\theta'}=1$.
Recall the integration by parts formula
\begin{multline}
\label{e:ibp1}
{\color{black}\bE} \left\langle D Y, X \tfrac{u}{\gamma} \right\rangle_H = {\color{black}\bE} \left[\left\langle D(XY), \tfrac{u}{\gamma} \right\rangle_H - \left\langle D X, Y \tfrac{u}{\gamma} \right\rangle_H \right]
\\
= {\color{black}\bE} \left[ X Y \mathop{\delta}\left(\tfrac{u}{\gamma}\right) - Y \left\langle D X, \tfrac{u}{\gamma} \right\rangle_H \right];
\end{multline}
it follows
\begin{align*}
\left|\bE\left\langle D Y, X \tfrac{u}{\gamma} \right\rangle_H \right| \le& \bE\left[ |Y| \, \left|X \mathop{\delta}\left(\tfrac{u}{\gamma}\right) -  \left\langle D X, \tfrac{u}{\gamma} \right\rangle_H\right|\right]
\end{align*}
and by H\"older's inequality (in the sequel we exploit the assumption $\theta > p'$, that is equivalent to $\theta' < p$)
\begin{align*}
\le& c \, \bE[|Y|^{\theta'}]^{1/\theta'} \, \left(\left[ \bE \left|X \mathop{\delta}\left(\tfrac{u}{\gamma}\right) \right|^{\theta} \right]^{1/\theta} +   \left[ \bE \left| \left\langle D X, \tfrac{u}{\gamma} \right\rangle_H\right|^{\theta}\right]^{1/\theta}\right)
\\
\le& c \, \bE[|Y|^{\theta'}]^{1/\theta'} \, \left(\left[ \bE |X|^{p} \right]^{1/p} 
\left[ \bE \left| \delta \left(\tfrac{u}{\gamma} \right) \right|^{\frac{p \theta}{p - \theta}} \right]^{\frac{p - \theta}{p \theta}}
+
\left[ \bE \|D X\|^{p}_H \right]^{1/p} 
\left[ \bE \left\| \tfrac{u}{\gamma} \right\|_H^{\frac{p \theta}{p - \theta}} \right]^{\frac{p - \theta}{p \theta}} \right)
\\
\le& c \, \bE[|Y|^{\theta'}]^{1/\theta'} \, \left(  \|X\|_{\bD^{1,p}} \, \left\| \tfrac{u}{\gamma} \right\|_{\textcolor{black}{D_{\frac{p \theta}{p - \theta}}(\delta)}}  \right)
\end{align*}
Recall the bound in \eqref{e:boundq}.
Then we shall require
\begin{align*}
1 < \frac{p \theta}{p - \theta} < \frac{n}{2} \qquad \Leftrightarrow \qquad \theta < \frac{np}{n+2p}.
\end{align*}
Notice that $p > \frac{n}{n-2}$ implies that $\frac{np}{n+2p} > 1$.
\end{proof}

We return to formula \eqref{e:2.3}. \textcolor{black}{Previous lemma guarantees the well posedness of the last term in \eqref{e:2.3}. }
Proceeding now in the same way we did in previous subsection, an application of Fubini's theorem implies that
\begin{align*}
F_X(r) =&  \int_{0}^r  \bE \left[ \delta\left(X \tfrac{u}{\gamma}\right)  \uno_{\{g > y\}}  \right] \, {\rm d}y.
\end{align*}
We can finally state the main result in this section. 

\begin{proposition}\label{p2.4}
Let $X$ belongs to $\bD^{1.p}$, $p > \frac{n}{n-2}$. Then there exists the derivative $f_X(r) = F_X'(r)$ and it is equal to 
\begin{align}\label{e:deriv-F}
F_X'(r) = \bE \left[ \delta\left(X \tfrac{u}{\gamma}\right)  \uno_{\{g > r\}}  \right].
\end{align}
Moreover, $f_X(r)$ is a continuous and bounded function and there exists a constant $c > 0$ such that
\begin{align}\label{e2.5}
|F'_X(r)| \le c \|X\|_{\bD^{1,p}}.
\end{align}
\end{proposition}

\begin{proof}
The proof of the first part is a straightforward extension of the computation of previous section, by taking into account the integrability of $X \frac{u}{\gamma}$ provided in Lemma \ref{e:2.3}.

We consider further the estimate \eqref{e2.5}. 
By the integration by parts formula for Malliavin derivative,
\begin{align*}
|F_X'(r)| \le \bE \left[ \left( |X \delta\left(\tfrac{u}{\gamma}\right) | + | \langle DX, \frac{u}{\gamma} \rangle_H | \right)  \uno_{\{g > r\}}  \right]
\end{align*}
and the thesis follows by H\"older's inequality and Hypothesis \ref{hp:1}.
\end{proof}

Actually, the existence of a continuous density for the functional $g$ implies that we can write formula \eqref{e:i2} as follows (we use a probabilistic notation, since it seems more expressive)
\begin{align}\label{e:sec-vers-der}
F_X(r) =  \bE[X \uno_{\{g < r\}}] = \int_{-\infty}^r \bE[X \mid g = s] f_1(s) \, {\rm d}s 
\end{align}
Therefore, by comparing with the results in Proposition \ref{p2.4}, we obtain that the identity
\begin{align}\label{e:ter-vers-der}
f_X(s) = F'_X(s)= \bE[X \mid g = s] f_1(s)
\end{align}
holds for almost every $s$ and, since the left-hand side is continuous, we conclude that there exists a continuous version of the function $s \mapsto \bE[X \mid g = s] f_1(s)$.

{
Notice that expression \eqref{e:ter-vers-der} for $F_X'(r)$ is more significant than \eqref{e:deriv-F}. In particular it provides "a candidate" to be the surface measure. As we will formally prove in what follows this candidate is given by $f_1(r)\, {\rm d}\mu$, where $f_1(r)$ is the density function of $g$. This also highlight the dependence of the surface measure by the kind of functional $g$ we consider.
}

\subsection{The surface measure}\label{subs2.3}

The results in this section mimic the construction in \cite{DaPrato2018, Bonaccorsi2018} and we shall skip some minor detail. Notice however that these papers only address the Hilbert setting, while we work in the Banach space $E$.
\\
Let us notice that on the probability space $(E, \cE, \mu)$, identity \eqref{e:ter-vers-der} formally reads
\begin{align*}
F'_X(r) = \int_{\{g  = r\}} X(x) f_1(r) \, \mu({\rm d}x).
\end{align*}
We are interested in proving that there exists a surface measure $\sigma_r$ on the boundary surface $\{g = r\}$ such that previous expression simplifies to
\begin{align*}
F'_X(r) = \int_{\{g  = r\}} X(x) \, \sigma_r({\rm d}x).
\end{align*}
In order to achieve this results, we need to extend previous construction to the class of functionals $X \in UC_b$.

Since functions in $UC_b$ can be uniformly approximated by elements in $UC_b^1$ (see \cite[Section 2.2]{DaPrato2002}), for every $X \in UC_b$ there exists a sequence $X_n \in UC_b^1$ such that $X_n \to X$.
Moreover, since $UC_b^1 \subset \bD^{1,p}$ for every $p$, results in previous section applies to the elements of the approximating sequence. 

\begin{proposition}
For every $X \in UC_b^1$, $F_X(r)$ is continuously differentiable and there exists a constant $c > 0$ such that
\begin{align}\label{e:stima-X-C1}
|f_X(r)| \le c \|X\|_{\infty}, \qquad \text{where $\|X\|_\infty$ is the sup-norm in $E$.}
\end{align}
\end{proposition}

\begin{proof}
Only \eqref{e:stima-X-C1} needs to be proven. By \eqref{e:ter-vers-der} we get
\begin{align*}
|f_X(r)| \le \left|\bE[X \mid g = r]\right| f_1(r) \le \bE|X| f_1(r) \le c \, \|X\|_{\infty}
\end{align*}
since we know, by Corollary \ref{c2.3}, that $f_1$ is a continuous and bounded function.
\end{proof}

By an approximation argument we obtain that the same result holds for $X \in UC_b$.

\begin{proposition}\label{p2.6}
For any $X \in UC_b$ the functional $F_X(r)$ is continuously differentiable and there exists a constant $c > 0$ such that
\begin{align}\label{e:stima-X}
|f_X(r)| \le c \|X\|_{\infty}.
\end{align}
\end{proposition}

We are finally in the position to conclude the proof of the main result of Theorem \ref{t1}.
For fixed $r$, consider a sequence $\varepsilon_n \to 0$ and define the family of measures
\begin{align*}
\sigma_n := \frac{1}{\varepsilon_n} \uno_{\{r < g \le r + \varepsilon_n\}} \mu.
\end{align*}
For any $X \in UC_b$ we have
\begin{align*}
\int_E X(x) \, \sigma_n({\rm d}x) = \int_E \frac{1}{\varepsilon_n} \uno_{\{r < g \le r + \varepsilon_n\}} X(x) \, \mu({\rm d}x)
= \frac{1}{\varepsilon_n} \left[ F_X(r+ \varepsilon_n) - F_X(r)\right];
\end{align*}
thanks to Proposition \ref{p2.6} we can pass to the limit in the above formula to get
\begin{align*}
\lim_{n \to \infty} \int_E X(x) \, \sigma_n({\rm d}x) = F_X'(r).
\end{align*}
By an application of the Prokhorov's theorem (see \cite[Corollary 8.6.3]{Bogachev2007}) we finally obtain that the sequence $\sigma_n$ converges to a measure $\sigma_r$ such that
\begin{align*}
F_X'(r) = \int_E X(x) \, \sigma_r({\rm d}x).
\end{align*}
Finally, by taking suitable approximations of $X = \uno_{\{|g-r|>\delta\}}$ we check that $\sigma_r$ is concentrated on $\{g = r\}$ and the proof is complete.

%\newpage
{\color{black}
\subsection{The integration by parts formula}

In this section we discuss the integration by parts formula on the level sets of the mapping $g$. Similar results have been obtained by \cite{Celada2014, Addona2018} with different techniques, see also \cite[Section 4]{Bonaccorsi2018}.

\begin{proposition}\label{pr:ibpf}
Let $r > 0$ be fixed. For any $X \in \bD^{1,p}$ and $h \in H$ it holds
\begin{equation}\label{e.ibp3}
\int_{\{g=r\}} X \langle D g, h \rangle_H \, \sigma_r({\rm d}x) = \int_{\{g < r\}} \left[ X W(h) - \langle DX,h \rangle_H \right] \, \mu({\rm d}x)
\end{equation}
where $W(h)$ is the Gaussian random variable defined in \eqref{e:wie1}. %by $W(h) = \int_0^1 \langle h(t), {\rm d}B(t) \rangle_{\bR^n}$.
\end{proposition}

\begin{proof}
The starting point is the integration by parts formula (compare \eqref{e:ibp1})
\begin{align}
\label{e:ibp2}
\bE \left[X \left\langle D Y, h \right\rangle_H \right] 
= \bE \left[ X Y \mathop{\delta}\left(h\right) - Y \left\langle D X, h \right\rangle_H \right];
\end{align}
which holds for random variables $X$ and $Y$ in the domain $\bD^{1,p}$ of the Malliavin derivative and $h \in H$.
\\
In a sense, we aim to apply this formula to the random variable $Y = \uno_{\{g < r\}}$, but this cannot be obtained directly due to the lack of regularity of this mapping.
We thus approximate $Y$ by the following procedure.

Let 
\begin{align*}
\theta_\varepsilon(a) = \displaystyle \frac{1}{\varepsilon} \int_{a}^{+\infty} \uno_{(r-\varepsilon, r)}(s) \, {\rm d}s, \qquad a > 0;
\end{align*}
$\theta_\varepsilon$ is a Lipschitz continuous function, hence the mapping $Y_\varepsilon = \theta_\varepsilon(g)$ is a smooth approximation of $Y$, in the sense that $Y_\varepsilon \to Y = \uno_{\{g \le r\}}$ in $L^2(E,\mu)$.

The right hand side of \eqref{e:ibp2}, with $Y_\varepsilon$ instead of $Y$, converges as $\varepsilon \downarrow 0$ to 
\begin{align*}
\int_{\{g < r\}} \left[ X W(h) - \langle DX,h \rangle_H \right] \, \mu({\rm d}x).
\end{align*}
On the other hand, we have
\begin{align*}
D Y_\varepsilon = \theta_\varepsilon'(g) \, D g = \frac{1}{\varepsilon} \uno_{(r-\varepsilon, r)}(g) \, D g
\end{align*}
hence
\begin{align*}
\bE \left[X \left\langle D Y_\varepsilon, h \right\rangle_H \right] = \frac{1}{\varepsilon} \int_H X \, \uno_{(r-\varepsilon, r)}(g) \langle D g,h \rangle_H \, \mu({\rm d}x).
\end{align*}
Proceeding as in Section \ref{subs2.3} we notice that
$\frac{1}{\varepsilon}\uno_{(r-\varepsilon, r)}(g) \, \mu$ converges to the measure $\sigma_r$ concentrated on $\{g = r\}$. We have thus proved the thesis.
\end{proof}

\begin{remark}
\em In the special case $X=1$, formula \eqref{e.ibp3} reads
\begin{align*}
F_{W(h)}(r) = \int_{\{g < r\}} W(h) \, \mu({\rm d}x) = \int_{\{g=r\}}  \langle D g, h \rangle_H \, \sigma_r({\rm d}x), \qquad h \in H.
\end{align*}
This is a sort of {\em divergence theorem} (in infinite dimensions) for the vector $h$; we remark that similar results are already present in the literature, compare for instance \cite{Goodman1972}.  
\end{remark}

\begin{remark}
\em Let us further notice that $D g$ is explicitly known (see formula \eqref{e.dg}), hence
\begin{align*}
\langle D g, h \rangle_H = \sum_{i=1}^n \int_0^1 h_i(s) \int_s^1 B^i(t) \, {\rm d}t \, {\rm d}s
\end{align*}
Let us define
\begin{align*}
\tilde h(t) = \int_0^1 (1 - t \vee r) h(r) \, {\rm d}r, \qquad t \in (0,1).
\end{align*}
Then it holds
\begin{align*}
\langle D g, h \rangle_H = %\mathop{\rm Tr}(
W(\tilde h)
%)
.
\end{align*}
\end{remark}

}

\section{Existence of the surface measure for sets defined by the solution of gradient systems}\label{sez5}

In this section we extend previous results to cover the case of a (multidimensional) gradient system SDE (see \cite{Hasminskii1980}). % with additive noise.
Let $V: \bR^n \to \bR$ be a potential energy function; % and $a(x) = - \nabla V(x)$.
we assume that
\begin{equation}
\label{e5.4}
V \in C^3_b(\bR^n),
\end{equation}
i.e., it is continuous and bounded together with its first three derivatives.
\\
Then we define $u$ to be the solution of the following equation:
\begin{align}\label{e5.1}
{\rm d}u(t) = - \nabla V(u(t)) \, {\rm d}t + {\rm d}B(t), \qquad u_0 = 0 \in \bR^n
\end{align}
Under our assumptions, the solution $u$ belongs to $L^2(E,\mu;E)$ (notice that we can solve the equation in a pathwise sense).

Recall from Corollary \ref{c2.3} that $g(B)$ has a density function $f_1(r)$ that is continuous and bounded for $r   > 0$.
In this section we aim to study the same property for the random variable $g(u)$, where $u$ is the solution of the equation \eqref{e5.1}.

\begin{theorem}\label{t5.1}
{Let $n \ge 3$}. The cumulative distribution function of $g(u)$ admits a probability density function
\begin{align*}
\varphi_1(r) = \int_{\{g(B) = r\}} \rho_1(B)^{-1} \, \sigma_r({\rm d}x)
\end{align*}
where for every process $h \in L^2(E,\mu;E)$ we let $\rho_1(h)$ be the Girsanov's density defined by
\begin{align}\label{e5.3}
\rho_1(h) = \exp\left( \int_0^1 \langle \nabla V(h(s)), {\rm d}B(s) \rangle - \frac12 \int_0^1 |\nabla V(h(s))|^2 \, {\rm d}s \right),
\end{align}
and the support of $\sigma_r$ is concentrated on $\{g(B) = r\}$.
\end{theorem}

\subsection{Change of measure}

First, we notice that our assumptions on $V$ implies that
\begin{align*}
\sup_{t \in [0,1]} \bE\left( \exp\left( |\nabla V(u(t))|^2 \right) \right) < +\infty,
\end{align*}
therefore by \cite[Theorem 10.14 \& Proposition 10.17]{DaPrato1992} we get that 
the process 
\begin{align*}
u(t) = B(t) - \int_0^t \nabla V(u(s)) \, {\rm d}s
\end{align*}
is a Brownian motion in $(E, \cE, \nu)$, where $\nu$ is a centered Gaussian measure such that
\begin{align*}
{{\rm d}\nu} = \rho_1(\cdot) \, {{\rm d}\mu}.
\end{align*}
Let $F: E \to \bR$ be a bounded and Borel function; then we have  that
\begin{align*}
\bE_\nu[F({u})] = \bE_\mu[F(B)].
\end{align*}

\begin{lemma}
\label{l5.1}
The following representation of $\rho_1(u)$ holds:
\begin{equation}
\label{e5.2}
\rho_1(u) = \exp\left( V(u(1)) + \frac12 \int_0^1 |\nabla V(u(t)) |^2 \, {\rm d}t - \frac12 \int_0^1 \mathop{\rm Tr}(\nabla^2 V(u(t))) \, {\rm d}t 
 \right).
\end{equation}
\end{lemma}

\begin{proof}
Let us compute the It\^o differential of $V(u)$:
\begin{align*}
{\rm d}V(u(t)) &= \langle \nabla V(u(t)), [-\nabla V(u(t)) \, {\rm d}t + {\rm d}B(t)] \rangle + \frac12 \mathop{\rm Tr}(\nabla^2 V(u(t))) \, {\rm d}t
%\sum_{i=1}^n \partial_{x_i}V(X(t)) \, {\rm d}X_i(t) + \frac12 \sum_{i,j=1}^n \partial_{x_i} \partial_{x_j} V(X(t)) \, {\rm d}\langle X \rangle_{ij}(t)
%\\
%&= \left[ - \sum_{i=1}^n a_i(X(t))^2 - \frac12 \sum_{i=1}^n \partial_{x_i} a_i(X(t)) \right] \, {\rm d}t - \sum_{i=1}^n a_i(X(t)) \, {\rm d}B^i(t)
\end{align*}
Therefore, using the integral form of previous differential and recalling that $u(0) = 0$, we get
\begin{align*}
V(u(1)) = - \int_0^1 |\nabla V(u(t)) |^2 \, {\rm d}t +
\int_0^1 \langle \nabla V(u(s)), {\rm d}B(s) \rangle  + \frac12 \int_0^1 \mathop{\rm Tr}(\nabla^2 V(u(t))) \, {\rm d}t 
\end{align*}
We substitute this expression in \eqref{e5.3} to get the thesis.
\end{proof}

\begin{proposition}\label{p5.1}
The mapping $x \mapsto \rho_1(B)^{-1}(x)$ belongs to $UC_b$.
\end{proposition}

\begin{proof}
By lemma \ref{l5.1}  we can write
\begin{align*}
\rho_1(B)^{-1} = \exp\left[-\left( V(B(1)) + \frac12 \int_0^1 |\nabla V(B(t)) |^2 \, {\rm d}t - \frac12 \int_0^1 \mathop{\rm Tr}(\nabla^2 V(B(t))) \, {\rm d}t  \right)\right].
\end{align*}
Then the assumption that $V \in C^3_b(\bR^n)$ implies that $\rho_1(B)^{-1}$ is bounded. Now, we exploit that $B$ is the canonical Brownian motion on the Wiener space $(E, \cE, \mu)$, hence
\begin{align*}
B(t)(x) - B(t)(y) = x(t) - y(t);
\end{align*}
notice again that the assumption on $V$ implies that the mappings on $E$ defined by 
\begin{align*}
x \mapsto V(x(1)), \qquad x \to \int_0^1 {\rm Tr} [\nabla^2 V(x(t))] \, {\rm d}t, \qquad x \to \int_0^1 |\nabla V(x(t))|^2 \, {\rm d}t
\end{align*}
are Lipschitz continuous. Therefore, if $\|x-y\|_\infty < \delta$, then $\left| \rho_1(B)^{-1}(x) - \rho_1(B)^{-1}(y) \right| \le e^{3L\delta}$ and the proof is complete.
\end{proof}

\subsection{The main result}
We have now all the ingredients to prove Theorem \ref{t5.1}.
The proof of the existence of the density for $g(u)$ can be obtained as a corollary to %Proposition \ref{p2.4} meglio dire 
the results of Section \ref{dif_sec}. 
To see this, we propose the following computation.
\\
Using Girsanov's transform we have
\begin{align*}
\mu(g(u) \le r) =& \bE_\mu[\uno_{\{g(u) \le r\}}] = \bE_\nu[\uno_{\{g(u) \le r\}} \, \rho_1(u)^{-1}]
\\
=& \bE_\mu[\uno_{\{g(B) \le r\}} \, \rho_1(B)^{-1}].
\end{align*}
More generally, it holds
\begin{align*}
\Phi_X(r) = \int_{\{g(u) \le r\}} X(x) \, \mu({\rm d}x) = F_{X \rho(B)^{-1}}(r).
\end{align*}

\begin{lemma}
The random variable $g(u)$, defined on the space $(E, \cE, \mu)$ with values in $\bR$, admits a probability density function with respect to the Lebesgue measure that is continuous and bounded.
\end{lemma}

\begin{proof}
Using Proposition \ref{p5.1} we are able to apply Theorem \ref{t1} to obtain that the distribution function $\Phi_1(r)$ of $g(u)$ admits a derivative
\begin{align*}
\varphi_1(r) = f_{\rho_1(B)^{-1}}(r) = \int_E \rho_1(B)^{-1} \, \sigma_r({\rm d}x)
\end{align*}
where, as stated in Theorem \ref{t1} the support of the measure $\sigma_r$ is concentrated on $\{g(B) = r\}$.
Now, the thesis follows from Proposition \ref{p2.4}.
\end{proof}

Next, we prove that there exists a surface measure on $\{g(u) = r\}$ for $r > 0$ that is the restriction of the Gaussian measure $\mu$ to the given surface.
%On the other hand, it holds that the measure
%$\theta_n = \rho_1(u) \sigma_n$ 
Proceeding as in Section \ref{sez2} we obtain that
%Let $\Phi_X(r) = \nu[ X \uno_{\{g(u) \le r\}}]$. It holds that $\Phi_X(r) = F_{X \rho(B)^{-1}}(r)$.
%In particular, 
\begin{align*}
\frac{1}{\epsilon_n}\left[\Phi_X(r+\epsilon_n) - \Phi(r)\right] = \int_E X(x) \frac{1}{\epsilon_n} \uno_{\{ r < g(u) \le r+\epsilon_n \}} \, \mu({\rm d}x)
\end{align*}
and we can pass to the limit in previous formula, since the left hand side converges to $\Phi'_X(r)=F'_{X \rho_1(B)^{-1}}(r)$ by Proposition \ref{p2.6}.
\\
Therefore, by mimicking the procedure in Section \ref{sez2} we get that the sequence of measures
\begin{align*}
\theta_n := \frac{1}{\epsilon_n} \uno_{\{ r < g(u) \le r+\epsilon_n \}} \, \mu({\rm d}x)
\end{align*}
converges to a measure $\theta_r$ and this measure is concentrated on $\{ g(u) = r\}$.
\\
In particular, 
\begin{align*}
\varphi_1(r) = \int_{\{ g(u) = r\}} \theta_r({\rm d}x) = \theta_r(\{ g(u) = r\}).
\end{align*}

\vfill\break

\end{document}